\documentclass[12pt]{article}
\usepackage{amsmath, amsthm}
\usepackage[top=4cm,left=3cm,right=3cm,bottom=4cm,headheight=13.6pt]{geometry}\usepackage{amsfonts}
\usepackage{amssymb}
\usepackage{fancyhdr}
\usepackage{enumerate}
\usepackage{amsmath,amsthm}
\usepackage[dvips]{graphicx}
\usepackage{color}
\usepackage{eepic}
\usepackage{epic}
\usepackage[dvips]{rotating}
\usepackage{sectsty}
\usepackage{wrapfig}

\theoremstyle{plain}
\newtheorem{theorem}{Theorem}[section]
\newtheorem{lemma}[theorem]{Lemma}

\theoremstyle{definition}

\newtheorem{remark}[theorem]{Remark}

\theoremstyle{remark}

 \numberwithin{equation}{section} 

\begin{document}
\title{On Blow-up  of A Reaction Diffusion System Coupled in Both Equations and Boundary Conditions}

\author{Maan A. Rasheed and Miroslav Chlebik}
\maketitle

\abstract 
We study the blow up solutions of a semilinear reaction diffusion system coupled in both equations and boundary conditions. The main purpose is to understand how the reaction terms and the absorption terms affect the blow-up properties. We derive the lower and upper bound for the blow-up rate, and find the blow-up set under certain assumptions.
\section{Introduction}
In this paper, we consider the following parabolic system
\begin{equation}\label{B20} \left. 
\begin{array}{lll}
u_t=\Delta u+\lambda_1e^{v},& v_t=\Delta v+\lambda_2 e^{u}, &(x,t)\in B_R \times (0,T), \\
 \frac{\partial u}{\partial \eta}=e^{v},&\frac{\partial v}{\partial \eta}=e^{u},&(x,t)\in\partial B_R \times (0,T),\\
u(x,0)=u_0(x),&v(x,0)=v_0(x),& x \in {B}_R, \\ 
\end{array}\right\} \end{equation}
where $\lambda_1,\lambda_2>0,$ $u_0,v_0$ are nonnegative, radial nondecreasing, smooth functions and satisfy the conditions
\begin{equation}\label{B23}\left. \begin{array}{lll}
\frac{\partial u_0}{\partial \eta}=e^{v_0},&\frac{\partial u_0}{\partial \eta}=e^{u_0},& x\in \partial B_R,\\
\Delta u_0+e^{v_0} \ge 0,&\Delta v_0+ e^{u_0} \ge 0,&x \in \overline{B}_R,\\
u_{0r}(|x|)\ge 0, & v_{0r}(|x|)\ge 0, &x \in \overline{B}_R.
 \end{array}\right\} \end{equation}
 
The problems of semilinear systems coupled in both equations and boundary conditions have been studied very extensively over past years in case the reaction terms and boundary conditions are of power type functions, for instance in \cite{52}, it was considered the solutions of the following system 
 \begin{equation}\label{B25} \left. 
 \begin{array}{lll}
u_t=u_{xx}+v^{p_1},& v_t=v_{xx}+u^{p_2},&(x,t)\in(0,1) \times (0,T), \\
u_x(1,t)=v^{q_1},& v_x(1,t)=u^{q_2}, &t\in(0,T),\\ 
u_x(0,t)=0,&  v_x(0,t)=0,&   t\in(0,T),\\ 
u(x,0)=u_0(x),&v(x,0)=v_0(x),& x\in [0,1],\\
\end{array}\right\}
\end{equation} 
where $p_1,p_2,q_1,q_2 >0,$ and $u_0,v_0$ are radial nondecreasing, positive smooth functions satisfying the conditions
$$u_{0x}(0)=v_{0x}(0)=0,\quad u_{0x}(1)=v^{q_1}_0(1),\quad v_{0x}(1)=u^{q_2}_0(1).$$
It was shown that if $$\max\{p_1p_2,p_1q_2,p_2q_1,q_1q_2\}\le 1,$$ then the solutions of problem (\ref{B25}) exists globally, otherwise every solution blows up in finite time. Moreover, the blow-up occurs only at $x=1$ and the blow-up rate estimates take the following form
$$C_1(T-t)^{-\alpha}\le u(1,t)\le C_2(T-t)^{-\alpha}, \quad t\in(0,T),$$ $$C_3(T-t)^{-\beta}\le v(1,t)\le C_4(T-t)^{-\beta}, \quad t\in(0,T),$$ where $$\alpha=\alpha(p_1,p_2,q_1,q_2),~\beta=\beta(p_1,p_2,q_1,q_2).$$

In \cite{MZ}, it was considered the critical exponents for a system of heat equations with inner absorption reaction terms and coupled boundary conditions of exponential type, namely
  \begin{equation}\label{B28} \left.
\begin{array}{lll}
  u_t=\Delta u-a_1e^{p_1u}, &v_t=\Delta v-a_2e^{p_2v},& (x,t)\in \Omega\times (0,T), \\
  \frac{\partial u}{\partial \eta}=e^{q_1v},&\frac{\partial v}{\partial \eta}=e^{q_2u}, & (x,t)\in \partial\Omega \times(0,T),\\   
u(x,0)=u_0(x),& v(x,0)=v_0(x),& x\in \Omega,\\
\end{array} \right\} \end{equation}
where $\Omega$ is a bounded domain with smooth boundary,~$p_1,p_2\ge 0,q_i,a_i> 0,i=1,2,$ $u_0,v_0$ are nonnegative functions that satisfy
$$\frac{\partial u_0}{\partial \eta}=e^{q_1v_0},\quad\frac{\partial u_0}{\partial \eta}=e^{q_2u_0},\quad x\in \partial B_R.$$
It was shown that if $$1/\tau_1>0,~\mbox{or}~ 1/\tau_2>0,$$ where
$$\tau_1=\frac{q_1+\frac{1}{2}p_2}{q_1q_2-\frac{1}{4}p_1p_2},\quad \tau_2=\frac{q_2+\frac{1}{2}p_1}{q_1q_2-\frac{1}{4}p_1p_2},$$ then the solutions of problem (\ref{B28})  with large initial data blow up in finite time.

The main purpose of this paper is to derive the upper and lower blow-up rate estimates for problem (\ref{B20}) and to study the blow-up set under some restricted assumptions.     

\section{Preliminaries}
Since the system (\ref{B20}) is uniformly parabolic, also the reaction and the boundary conditions terms are smooth functions and the initial data satisfy the compatibility conditions, therefore, the local existence and uniqueness of the classical solutions of problem (\ref{B20}) are known by standard parabolic theory (see \cite{37}). On the other hand, for any initial data $(u_0,v_0),$ the solution of this system has to blow up in finite time and the blow-up set contains the boundary ($\partial B_R$), and that due to the comparison principle \cite{21} and the known blow-up results of problem (\ref{B20}), where $\lambda_1=\lambda_2=0,$  which has been studied in \cite{18}. 

The next lemma shows the properties of the classical solutions of problem (\ref{B20}). We denote for simplicity $u(r,t)=u(x,t).$
\begin{lemma}\label{Cv} Let $(u,v)$ be a classical solution to problem (\ref{B20}). Then
\begin{enumerate}[\rm(i)] \item (u,v) is radial and $u,v>0$ in $\overline B_R \times (0,T).$ 
\item $u_r ,v_r \ge 0$ in $[0,R]\times (0,T).$
 \item   $u_t ,v_t > 0,$ in $\overline{B}_R\times (0,T).$ 
\end{enumerate}
\end{lemma}

Next, we prove the following lemma, which shows the relation between $u$ and $v.$  
\begin{lemma}\label{askt}
Let $(u,v)$ be a solution to problem (\ref{B20}), there exist $M>1$ such that \begin{equation}\label{B31}
  e^{v}\le Me^{u}, \quad e^{u} \le Me^{v},\quad (x,t)  \in \overline{B}_R \times [0,T).\end{equation} 
\end{lemma}
\begin{proof}
Let $$J(x,t)=Me^{u(r,t)}-e^{v(r,t)}, \quad (x,t) \in B_R\times (0,T),\quad r=|x|.$$ 
A direct calculation shows
\begin{eqnarray} J_t&=&Me^uu_t-e^vv_t,\nonumber \\
 \label{B32} J_r&=& Me^uu_r- e^vv_r,\\
J_{rr}&=&Me^uu_{rr}+Me^uu^2_r-e^v v_{rr}-e^vv_r^2.\nonumber \end{eqnarray}
Thus \begin{eqnarray*} J_t-J_{rr}-\frac{n-1}{r}J_r&=&Me^uu_t-e^vv_t-Me^uu_{rr}-Me^{u}{u_r^2}+e^v v_{rr}+e^vv_r^2\\ &&-\frac{n-1}{r}Me^uu_r+\frac{n-1}{r}e^vv_r\\
&=&Me^u[u_t-u_{rr}-\frac{n-1}{r}u_r]-e^v[v_t-v_{rr}-\frac{n-1}{r}v_r]\\ &&-Me^{u}{u_r^2}+e^vv_r^2\\ &=&Me^u[\lambda_1 e^v]-e^v[\lambda_2 e^u]-Me^{u}{u_r^2}+e^vv_r^2.\end{eqnarray*}
From (\ref{B32}), it follows that \begin{eqnarray*} u_r&=&\frac{1}{Me^u}[v_re^v+J_r],\\ 
u^2_r&=&\frac{1}{M^2e^{2u}}[v_r^2e^{2v}+2e^vv_rJ_r+J_r^2].\end{eqnarray*}
Therefore, $$J_t-\Delta J=(\lambda_1M-\lambda_2)e^{u+v}+[e^v-\frac{e^{2v}}{Me^{u}}]v_r^2-[ \frac{2e^v}{Me^{u}}v_r+\frac{1}{Me^{u}} J_r]J_r.$$
Clearly, $$e^v-\frac{e^{2v}}{Me^{u}}=e^v\frac{J}{Me^u}.$$ Therefore, the last equation can be rewritten as follows:
$$J_t-\Delta J-bJ_r-cJ =(\lambda_1M-\lambda_2)e^{u+v}\ge 0,\quad (x,t)\in B_R\times (0,T)$$ provided $M>\lambda_2/\lambda_1,$
where, $$b=-[ \frac{2e^v}{Me^{u}}v_r+\frac{1}{Me^{u}} J_r],~c=\frac{e^v}{Me^u}v_r^2.$$
It clear that, $b,c$ are continuous functions and $c$ is bounded in $B_R \times (0,T^*),$ for  $T^*<T.$  

Moreover, \begin{eqnarray*}\frac{\partial J}{\partial \eta}|_{x\in \partial B_R} &=&[Me^uu_r- e^vv_r]\\ &=& Me^{u+v}-e^{u+v}=[M-1]e^{u+v}>0, 
\end{eqnarray*}
  and $$J(x,0)=Me^{u_0}-e^{v_0}\ge 0,\quad x \in \overline{B}_R$$ provided $M$ is large enough. 
  
  From above and Proposition \cite{21}, it follows that $$J\ge 0,\quad \mbox{in}\quad \overline{B}_R\times [0,T).$$ 
  Similarly, we can show that the function $H=Me^v-e^u$ is nonnegative in $\overline{B}_R \times [0,T).$
\end{proof}
\section{Blow-up Rate Estimates}
In this section we consider the upper and lower blow-up rate estimates of solutions for problem (\ref{B20}) with (\ref{B23}).
\begin{theorem}\label{Bc}
Let $u$ be a blow-up solution solution of problem (\ref{B20}) with (\ref{B23}), $\lambda_1=\lambda_2=\lambda,$ $T$ is the blow-up time. Assume that $u_0,v_0$ satisfy  \begin{equation}\label{xa} u_{0r}(r)-\frac{r}{R}e^{v_0(r)}\ge 0,\quad v_{0r}(r)-\frac{r}{R}e^{u_0(r)}\ge 0, \quad r \in [0,R].\end{equation}
 Then there is a positive constant $c$ such that
$$\log{c}-\frac{1}{2}\log (T-t)\le u(R,t),\quad \log{c}-\frac{1}{2}\log (T-t)\le v(R,t),\quad t \in (0,T).$$
\end{theorem}
\begin{proof}
Define the functions $J_1,J_2$ as follows:  
$$J_1(x,t)=u_r(r,t)-\frac{r}{R}e^{v(r,t)},\quad J_2(x,t)=v_r(r,t)-\frac{r}{R}e^{u(r,t)}.$$
A direct calculation shows  
\begin{eqnarray*} J_{1t}&=&u_{rt}-\frac{r}{R}e^{v}[v_{rr}+\frac{n-1}{r}v_r+\lambda e^u],\\
J_{1r}&=&u_{rr}-\frac{r}{R}e^vv_r-\frac{1}{R}e^v,\\
J_{1rr}&=&[u_{rt}-\frac{n-1}{r}u_{rr}+\frac{n-1}{r^2}u_r-\lambda e^vv_r]\\ &&-\frac{r}{R}[e^vv_{rr}+e^vv_r^2]-\frac{2}{R}e^vv_r.\end{eqnarray*}
From above it follows that
\begin{eqnarray*} J_{1t}-J_{1rr}-\frac{n-1}{r}J_{1r}=-\frac{n-1}{r^2}[u_r-\frac{r}{R}e^v]+\lambda e^v[v_r-\frac{r}{R}e^{u}]+\frac{r}{R}e^vv_r^2+\frac{2}{R}e^vv_r. \\
\end{eqnarray*}
Thus $$J_{1t}-\Delta J_{1}+\frac{n-1}{r^2}J_1-\lambda e^vJ_2=\frac{r}{R}e^vv_r^2+\frac{2}{R}e^vv_r\ge 0,$$ for $(x,t) \in B_R \times (0,T)\cap \{r>0\}.$ 

In the same way we can show that $$J_{2t}-\Delta J_{2}+\frac{n-1}{r^2}J_2-\lambda e^uJ_1\ge 0,\quad (x,t) \in B_R \times (0,T)\cap \{r>0\}.$$
Clearly, from (\ref{xa}), it follows that  $$J_1(x,0),~J_2(x,0)\ge 0 \quad x \in B_R.$$ And \begin{eqnarray*} &&J_1(0,t)=u_r(0,t)\ge 0,J_2(0,t)=v_r(0,t)\ge 0,\\ &&J_1(R,t)=J_2(R,t)=0, \quad t \in (0,T).\end{eqnarray*} Since, the supremums of the functions $\lambda e^u,\lambda e^v$ and $\frac{1-n}{r^2}$ (on $B_R\times (0,t]$ for $t<T$) are finite, therefore, from above and maximum principle, it follows 
$$J_1,J_2\ge 0,\quad (x,t)\in B_R \times (0,T).$$ Moreover, 
$$\frac{\partial J_1}{\partial \eta}|_{\partial B_R}\le 0.$$
This means 
$$(u_{rr}-\frac{r}{R}e^vv_r-\frac{1}{R}e^v)|_{\partial B_R}\le 0.$$ 
Thus $$u_t\le (\frac{n-1}{r}u_r+\lambda e^v+\frac{r}{R}e^vv_r+\frac{1}{R}e^v)|_{\partial B_R}.$$ Which implies that
 $$u_t(R,t)\le \frac{n-1}{R}e^{v(R,t)}+\lambda e^{v(R,t)}+e^{v(R,t)+u(R,t)}+\frac{1}{R}e^{v(R,t)},\quad t \in (0,T).$$
From the last inequality and Lemma \ref{askt}, it follows  
 $$u_t(R,t)\le \frac{n-1}{R}Me^{u(R,t)}+\lambda Me^{u(R,t)}+Me^{2u(R,t)}+\frac{M}{R}e^{u(R,t)},\quad t \in (0,T).$$
Thus, there exist a constant $C$ such that $$u_t(R,t)\le Ce^{2u(R,t)},\quad t \in (0,T).$$ Integrate this inequality from $t$ to $T$ and since $u$ blows up at $R,$ it follows 
$$\frac{c}{(T-t)^{\frac{1}{2}}}\le e^{u(R,t)},\quad t \in (0,T)$$  
or $$\log{c}-\frac{1}{2}\log (T-t)\le u(R,t),\quad t \in (0,T).$$
We can show in a similar way that 
$$\log{c}-\frac{1}{2}\log (T-t)\le v(R,t),\quad t \in (0,T).$$  
\end{proof}
Next, we consider the upper bounds 
\begin{theorem}\label{Bcc}
Let $u$ be a blow-up solution solution of problem (\ref{B20}), (\ref{B23}), $T$ is the blow-up time. Then there is a positive constant $C$ such that
$$u(R,t) \le \log{C}-\log{(T-t)},\quad v(R,t) \le \log{C}-\log{(T-t)},\quad t \in (0,T).$$
\end{theorem}
\begin{proof}
Define $$M(t)=\max_{\overline{B}_R}u(x,t),\quad N(t)=\max_{\overline{B}_R}v(x,t).$$
$M(t),N(t)$ are increasing in $(0,T)$ due to the $$u_t ,v_t> 0,\quad (x,t) \in \overline B_R\times (0,T).$$ For $0< z<t<T, x \in B_R,$  as in \cite{22}, the integral equation for problem (\ref{B20}) with respect to $u$ can be written as follows 
\begin{eqnarray*} u(x,t)&=&\int_{B_R} \Gamma(x-y,t-z)u(y,z)dy+\lambda_1 \int_z^t \int_{B_R} \Gamma(x-y,t-\tau)e^{v(y,\tau)}dy d\tau\\ &&+\int_z^t \int_{S_R} \Gamma(x-y,t-\tau)e^{v(y,\tau)}ds_y d\tau\\ &&-\int_z^t \int_{S_R} u(y,\tau)\frac{\partial \Gamma}{\partial \eta_y}(x-y,t-\tau)ds_y d\tau,\end{eqnarray*}  
where $\Gamma$ is the fundamental solution of the heat equation, which takes the form:
 \begin{equation}\label{op}\Gamma(x,t)=\frac{1}{(4\pi t)^{(n/2)}}\exp[-\frac {|x|^2}{4t}] .\end{equation} 
 Letting $x\rightarrow \partial B_R$ and using the jump relation, \cite{23}, for the fourth term on the right hand side of the last equation, we obtain 
\begin{eqnarray*} \frac{1}{2}u(x,t)&=&\int_{B_R} \Gamma(x-y,t-z)u(y,z)dy+\lambda_1\int_z^t \int_{B_R} \Gamma(x-y,t-\tau)e^{v(y,\tau)}dy d\tau\\ &&+\int_z^t \int_{S_R} \Gamma(x-y,t-\tau)e^{v(y,\tau)}ds_y d\tau\\ &&-\int_z^t \int_{S_R} u(y,\tau)\frac{\partial \Gamma}{\partial \eta_y}(x-y,t-\tau)ds_y d\tau,\end{eqnarray*}  for $x\in \partial B_R, 0<z<t<T.$

Since $u,v$ are positive and radial, it follows  \begin{eqnarray*} && \int_{B_R} \Gamma (x-y,t-z)u(y,z)dy >0,\\ && \int_z^t \int_{S_R}e^{v(y,\tau)}
\Gamma (x-y,t-\tau)d{s_y}d_\tau=\int_z^t e^{v(R,\tau)}[\int_{S_R}\Gamma (x-y,t-\tau)ds_y]d\tau.\end{eqnarray*}  Thus
 \begin{eqnarray*} \frac{1}{2}M(t) &\ge& \int_z^t e^{N(\tau)}[\int_{S_R}\Gamma (x-y,t-\tau)ds_y]d\tau \\
&&-\int_z^t M(\tau)[\int_{S_R} | \frac{\partial \Gamma}{\partial \eta _y} (x-y,t-\tau )| ds_y  ] d \tau, \quad x\in S_R, 0<z<t<T. \end{eqnarray*}
 It is known that (see \cite{23}) for  $0<t_2<t_2,$ these is $C^*>0$ such that 
 $$|\frac{\partial \Gamma}{\partial \eta_{y}}(x-y,t_2-t_1)|\le \frac{C^*}{(t_2-t_1)^\mu}\cdot \frac{1}{|x-y|^{(n+1-2\mu- \sigma ) }},\quad x,y \in S_R, \sigma\in (0,1).$$
Choose $1-\frac{\sigma}{2} < \mu <1,$ from \cite{23}, there exist $C_1>0$ such that
$$\int_{S_R}\frac{ds_y}{|x-y|^{(n+1-2\mu- \sigma ) }} <C_1.$$ Also, if $t_1$ close to $t_2,$ then there exist a constant $c$ such that  $$\int_{S_R} \Gamma (x-y,t_2-t_1)ds_y \ge \frac{c}{\sqrt{t_2-t_1}}.$$
  Thus  $$\frac{1}{2} M(t) \ge c \int_z^t \frac{e^{N(\tau)}}{\sqrt{t-\tau}}d\tau-C\int_z^t \frac{M(\tau)}{|t-\tau|^{\mu}}d \tau.$$
Since for $0<z<\tau< t <T,$ it is clear that $M(\tau) \le M(t),$ thus
\begin{equation}\label{d8}
\frac{1}{2}M(t)\ge c \int_z^t \frac{e^{N(\tau)}}{\sqrt{T-\tau}}d\tau-C^*_1 M(t)|T-z|^{1-\mu}.
\end{equation}
Taking $z$ so that  $C^*_1|T-z|^{1-\mu}=1/2,$ it follows
\begin{equation}\label{d10}
M(t)\ge c \int_z^t \frac{e^{N(\tau)}}{\sqrt{T-\tau}}d\tau \equiv A(t). \end{equation}
Clearly, $$A^{'}(t)=c\frac{e^{N(t)}}{\sqrt{T-t}}.$$ From Lemma \ref{askt}, there exist a constant  $k>1$ such that the last equation becomes
$$A^{'}(t)=\frac{c}{k}\frac{e^{M(t)}}{\sqrt{T-t}}\ge \frac{c}{k}\frac{e^{A(t)}}{\sqrt{T-t}},$$
which leads to $$\int_t^T \frac{dA}{e^{A}}\ge \int_t^T \frac{c}{k} \frac{d\tau}{\sqrt{T-\tau}}.$$
Clearly, $$A(T)=\lim_{t\rightarrow T}c \int_z^t \frac{e^{N(\tau)}}{\sqrt{T-\tau}}d\tau=c \int_z^t \lim_{\tau\rightarrow T}\frac{e^{N(\tau)}}{\sqrt{T-\tau}}d\tau=\infty.$$  
This leads to 
$$\frac{1}{e^{A(t)}}\ge \frac{2c}{k} \sqrt{T-t}.$$ Therefore, there exist a constant $C_0>0$ such that 
\begin{equation}\label{RT}e^{A(t)}\le \frac{C_0}{ \sqrt{T-t}}, \quad z<t<T.\end{equation}
On the other hand, for $t_0=2t-T$ (Assuming that $t$ is close to $T$),
$$A(t)\ge c\int_{t_0}^t \frac{e^{N(\tau)}}{\sqrt{T-\tau}}d\tau\ge c e^{N(t_0)} \int_{2t-T}^{t} \frac{1}{\sqrt{T-\tau}}d\tau=e^{N(t_0)}2c(\sqrt{2}-1)\sqrt{T-t} .$$
Combining the last inequality with (\ref{RT}), yields
$$\frac{C_0}{ \sqrt{T-t}}\ge e^{N(t_0)}2c(\sqrt{2}-1)\sqrt{T-t},$$ which leads to 
$$e^{N(t_0)}\le \frac{C_0}{c(\sqrt{2}-1)(T-t_0)}.$$ Thus there exist a constant $C$ such that 
 $$e^{N(t)}\le  \frac{C}{(T-t)},\quad 0 < t < T$$ or $$v(R,t) \le \log{C}-\log{(T-t)},\quad t \in (0,T).$$ In the same way we can show 
$$u(R,t) \le \log{C}-\log{(T-t)},\quad t \in (0,T).$$ 
 \end{proof}
\begin{remark} From Theorems \ref{Bc}, \ref{Bcc}, we conclude that, the upper blow-up rate estimates of problem (\ref{B20}) are coincident with the upper blow-up rate estimates of the zero Dirichlet problem for the semilinear system in (\ref{B20}), while the lower blow-up rate estimates of problems (\ref{B20}) are coincident with the lower blow-up rate estimates of problem (\ref{B20}), where $\lambda_1=\lambda_2=0$ (see \cite{7}).  
\end{remark}

 \section{Blow-up Set} We consider next the blow-up set for problem (\ref{B20}), under some restricted assumptions on $\lambda_1,\lambda_2.$
\begin{theorem}\label{}
Let  $(u,v)$ be a blow-up solution to problem (\ref{B20}). Assume that the following condition is satisfied 
\begin{equation}\label{FF} \lambda[4R^2(n+1)+1]\le \min \left\{ \frac{1}{C}, \frac{4(n+1)}{[R^2+4(n+1)T]}e^{-||u_0||_\infty},\frac{4(n+1)}{[R^2+4(n+1)T]}e^{-||v_0||_\infty} \right\}, \end{equation}
where $T$ is the blow-up time, $C$ is given in Theorem \ref{Bcc}, $\lambda=\max\{\lambda_1,\lambda_2\}.$ 
Then there exist a positive constant $A$ such that 
$$u(x,t) \le \log [\frac{1}{A(R^2-r^2)^2}],\quad v(x,t) \le \log [\frac{1}{A(R^2-r^2)^2}], $$ for $(x,t) \in B_R\times (0,T),$  \end{theorem}
 \begin{proof} 
 Define the functions $z_1,z_2$ as follows 
 \begin{equation}\label{Dz}\begin{array}{ll} z_1(x,t)=z_2(x,t)=\log \frac{1}{[Av(x)+B(T-t)]},&\quad(x,t)\in \overline{B}_R \times (0,T),\\

 \end{array} \end{equation}
 
  where $v(x)=(R^2-r^2)^2,\quad r=|x|,$ ~$B>0,$~ $A\ge \lambda.$
  
    A direct calculation shows: 
  \begin{equation}\label{Dv} \left. \begin{array}{ll}  z_{1t}-\Delta z_1-\lambda_1 e^{z_2}\ge z_{1t}-\Delta z_1-Ae^{z_2}\ge 0,& \mbox{in}~ {B}_R \times (0,T),\\ 
   z_{2t}-\Delta z_2-\lambda_2 e^{z_1} \ge z_{2t}-\Delta z_2-Ae^{z_1} \ge 0,& \mbox{in}~{B}_R \times (0,T) \end{array} \right\}  \end{equation} 
  provided $$B\ge A[4R^2(n+1)+1].$$ 
  Moreover, 
\begin{equation}\label{V1} \left. \begin{array}{lll} z_1(x,0)=\log \frac{1}{[Av(x)+BT]}\ge \log \frac{1}{[AR^4+BT]}\ge u(x,0) ,& x \in {B}_R, \\
 z_2(x,0)=\log \frac{1}{[Av(x)+BT]} \ge \log \frac{1}{[AR^4+BT]}\ge v(x,0) ,& x \in {B}_R \end{array} \right\} \end{equation}
 and  
 \begin{equation}\label{V2} \begin{array}{lll} 
z_1(R,t)=z_2(R,t)=\log \frac{1}{B(T-t)}\ge \log \frac{C}{(T-t)},& t \in (0,T) \\

\end{array} \end{equation}  provided 
$$B\le \min \left\{ \frac{1}{C}, \frac{4(n+1)}{R^2+4(n+1)T}e^{-||u_0||_\infty},\frac{4(n+1)}{R^2+4(n+1)T}e^{-||v_0||_\infty} \right\},$$

From (\ref{V1}), (\ref{V2}) and Theorem \ref{Bcc}, it follows that  
\begin{equation}\label{V3} \left. \begin{array}{lll} z_1(R,t)\ge u(R,t),& z_2(R,t) \ge v(R,t),& t \in (0,T),\\
 z_1(x,0) \ge u(x,0), & z_2(x,0) \ge v(x,0), & x\in {B}_R. \end{array} \right \} \end{equation} 
 
From  (\ref{Dv}), (\ref{V3}) and the comparison principle \cite{2}, it follows that $$z_1(x,t)\ge u(x,t), \quad z_2(x,t)\ge v(x,t),\quad (x,t)\in B_R \times (0,T).$$  Moreover, from (\ref{Dz})
\begin{equation}\label{yn} u(x,t) \le \log [\frac{1}{A(R^2-r^2)^2}],\quad v(x,t) \le \log [\frac{1}{A(R^2-r^2)^2}],\end{equation} for $(x,t) \in B_R\times (0,T).$  
\end{proof}
\begin{remark} From (\ref{yn}), we conclude that, for problem (\ref{B20}) with (\ref{FF}), any point  $x\in B_R$ cannot be a blow-up point, therefore, the blow-up occurs only at the boundary. This means, if $\lambda_1,\lambda_2$ are small enough, then the blow-up set  is the same as that of (\ref{B20}), where $\lambda_1=\lambda_2=0$ (see \cite{18}).
\end{remark}

\end{document}